\documentclass[12pt,draft]{amsart}
\usepackage{amssymb}
\usepackage{latexsym}
\usepackage{amsfonts}
\usepackage{amsmath}
\usepackage{color}
\usepackage{multicol}
\usepackage{lscape}

\newtheorem{theorem}{Theorem}[section]
\newtheorem{proposition}[theorem]{Proposition}
\newtheorem{lemma}[theorem]{Lemma}
\newtheorem{corollary}[theorem]{Corollary}

\theoremstyle{definition}

\newtheorem{example}[theorem]{Example}

\renewcommand{\leq}{\leqslant}
\renewcommand{\geq}{\geqslant}
\newcommand{\Z}{\mathbb Z}
\newcommand{\Q}{\mathbb Q}

\newcommand{\F}{\mathbb F}

\newcommand{\n}{\mathfrak n}
\newcommand{\bft}{\mathbf t}

\newcommand{\bfy}{\mathbf y}
\newcommand{\bfu}{\mathbf u}

\newcommand{\g}{\mathfrak g}

\newcommand{\fracfield}[1]{\textrm{Frac}(#1)}

\newcommand{\slg}{\mathfrak{sl}_2}
\newcommand{\dprod}[3]{#1\circ_{#3} #2}

\begin{document}

\title[]{The center and invariants of standard filiform Lie algebras}
\author{Vanderlei Lopes de Jesus}
\address[Lopes]{Departamento de Matem\'atica\\
Instituto de Ci\^encias Exatas\\
Universidade Federal de Minas Gerais\\
Av.\ Ant\^onio Carlos 6627\\
Belo Horizonte, MG, Brazil. Email:
{\rm\texttt{vanderleilopesbh@gmail.com}}}

\author{Csaba Schneider}
\address[Schneider]{Departamento de Matem\'atica\\
Instituto de Ci\^encias Exatas\\
Universidade Federal de Minas Gerais\\
Av.\ Ant\^onio Carlos 6627\\
Belo Horizonte, MG, Brazil. Email:
{\rm\texttt{csaba@mat.ufmg.br}}, URL: {\rm\texttt{schcs.github.io/WP/}}
}

\begin{abstract}
    This paper describes the centers 
    of the universal enveloping algebras and the invariant rings of the standard filiform Lie algebras over fields of characteristic zero and also over large enough prime characteristic.  We determine explicit generators 
    for the quotient fields and also a compact form for the generators for 
    the invariants rings. We prove several combinatorial results concerning the Hilbert series 
    of these algebras.
\end{abstract}

\date{\today}


\maketitle

\section{Introduction}\label{section1}

The problem of describing the center $Z(\g)$ of the universal enveloping algebra for nilpotent Lie algebras $\g$ has received quite a bit of 
attention. Dixmier~\cite{dixmier58} determined $Z(\g)$ for nilpotent Lie algebras of dimension up to~5, 
and Ooms in~\cite{Ooms1,Ooms2} extended this work to algebras of dimension at most~7. Considering the 
descriptions of the centers $Z(\g)$ for these small-dimensional nilpotent Lie algebras, it becomes apparent
that the most complicated structure is exhibited in the cases when $\g$ is standard filiform.
So we decided that the study of the center $Z(\g)$ for standard filiform Lie algebras $\g$ is worth special attention. 
In this paper we consider both the cases of characteristic zero and prime characteristic.

The {\em standard filiform Lie algebra} $\g(n+2)$ is the Lie algebra over a field 
$\F$ with basis 
$\{x,y_0,y_1,\ldots,y_{n}\}$ and nonzero brackets $[x,y_i]=y_{i-1}$ for 
$i\in\{1,\ldots,n\}$ (in particular, $[x,y_0]=0$ and $y_0$ is central). 
This is a nilpotent Lie algebra of dimension $n+2$ and nilpotency class $n+1$; such a Lie algebra 
is also referred to as a {\em Lie algebra of maximal class}. The center of $\g(n+2)$ is one-dimensional 
and is generated by $y_0$.  

As it turns out (and explained in Section~\ref{sec:stfil}), the center $Z(\g)$ of the standard filiform 
Lie algebra $\g=\g(n+2)$ (see Section~\ref{sec:stfil} for the notation) is equal to  
the algebra of polynomial solutions $f\in\F[y_0,\ldots,y_n]$ of the partial differential equation
\[ 
    \sum_{i\geq 0}y_i\frac{\partial f}{\partial y_{i+1}}=0.
\]
The operator $f\mapsto \sum y_i\partial f/\partial y_{i+1}$ is a locally nilpotent derivation,
also 
known as a Weitzenb\"ock derivation,
on the polynomial algebra $\F[y_0,\ldots,y_n]$ and was considered in~\cite{bed1,bed2,bed3,bed4,bed5,bed6}.
The same operator on the polynomial ring $\F[y_0,y_1,\ldots]$ on infinitely many generators is referred
to as the \emph{down operator}, see~\cite{freud}. The crucial observation made by~\cite{bed_arx,bed_ukrj} is 
that the action of this operator on linear polynomials can be extended to an irreducible representation 
of the simple Lie algebra $\slg$ and one can 
use the representation theory of $\slg$ to obtain more information on the center $Z(\g)$. In fact, 
considering the standard basis $\{e,f,h\}$ of $\slg$, the 
homogeneous elements of $Z(\g)$ are homogeneous polynomials in $y_0,y_1,\ldots,y_n$, which 
are $h$-eigenvectors and are annihilated by $e$. 
Thus, $Z_n$ is linked to the algebra of $\slg$-covariants of the 
binary form of degree~$n$.  

For $n\geq 1$, let $Z_n$ denote $Z(\g(n+2))$. If the characteristic of the field is zero, then $Z_n$ can 
be viewed as a graded subalgebra of the polynomial algebra $\F[y_0,y_1,\ldots,y_n]$. For $k\geq 0$, let 
$Z_{n,k}$ denote the degree-$k$ homogeneous component of $Z_n$. 

\begin{theorem}\label{th:th1}
    If $\F$ has characteristic zero, then the algebra $Z_n$ has Krull dimension $n$ and the following are valid.
    \begin{enumerate}
        \item The fraction field $\fracfield{Z_n}$ can be written as  
        \[
            \fracfield{Z_n}=\F(z_1,z_2,\ldots,z_n)=\F(w_1,w_2,\ldots,w_n)
        \] 
         where the elements $z_1,z_2,\ldots,z_n$ and 
        $w_1,w_2,\ldots,w_n$ are defined in Section~\ref{sec:expgens}.
        \item $Z_n\subseteq \F[y_0^{-1},z_1,z_2,\ldots,z_n]$ and $Z_n\subseteq 
        \F[y_0^{-1},w_1,\ldots,w_n]$. 
    \end{enumerate}
\end{theorem}

In Section~\ref{sec:expgens} we define an operation $\circ_d$ between two homogeneous elements of $Z_n$ 
that are also $h$-eigenvectors  for the generator $h$ of the Cartan subalgebra of $\slg$. 
We show that $Z_n$ is linearly spanned by elements of the form 
$y_0\circ_{d_1}y_0\circ_{d_2}\cdots \circ_{d_k} y_0$ 
and one can choose minimal algebra generating sets out of such elements. Using computer calculation, we explicitly determine such generating sets for $Z_{n}$ for $n\leq 6$ and for $n=8$ in the 
Appendix~\ref{app:gens}.  

In Section~\ref{sec:hilb}, we consider the Hilbert series 
\[
    H_n(t)=\sum_{k\geq 0} \dim Z_{n,k}t^k.
\]
It is interesting that the dimensions $\dim Z_{n,k}$ occur also in the context of representation theory
and combinatorics. 

\begin{theorem}\label{th:th2}
    Assuming that $\F$ has characteristic zero, that $n\geq 1$ and $k\geq 1$, the following are true for $\dim Z_{n,k}$:
    \begin{enumerate}
        \item $\dim Z_{n,k}$ coincides with the number of irreducible components 
        of the $\slg$-module $S^k(U_n)$ where $U_n$ is the $(n+1)$-dimensional irreducible representation 
        of $\slg$ and $S^k$ is the $k$-th symmetric power.
        \item $\dim Z_{n,k}$ is equal to the number of partitions 
        of $\lfloor kn/2\rfloor$ into $k$ blocks each of size at most~$n$ (permitting block size zero). 
        \item $\dim Z_{n,k}=\dim Z_{k,n}$.
    \end{enumerate}
\end{theorem}

Theorem~\ref{th:th2} follows from Theorem~\ref{th:deltand} (part~(1)) and 
from Theorem~\ref{th:part} (parts~(2) and~(3)).

By the Mauer--Weitzenb\"ock Theorem~\cite[Theorem~6.1]{freudbook}, $Z_n$ is finitely generated and so 
$H_n(t)$ is a rational function. 
The combinatorial description of $\dim Z_{n,k}$ in Theorem~\ref{th:th2} and the computations in 
\cite{zeilberg} make it possible to determine 
the Hilbert series $H_n(t)$ as rational expressions for $n\leq 18$. 
As in the situation considered by Almkvist~\cite{almkvist2,almkvist3}, one also has the following closed formula for $H_n(t)$. 

\begin{theorem}\label{th:th3}
    If $\F$ has characteristic zero, then, for $n\geq 1$, 
    \begin{equation}\label{eq:inteq}
        H_n(t)=\frac{1}{2\pi}\int_{-\pi}^\pi\frac{1+\exp(i\varphi)}{\prod_{k=0}^n(1-t\exp(i(n-2k)\varphi))}\,d\varphi.
    \end{equation}
\end{theorem}

The integral formula in Theorem~\ref{th:th3} appears in~\cite{almkvist2}, since in that situation, 
the counting argument also boils down to counting the partitions in Theorem~\ref{th:th2}(2). We include a 
proof in Section~\ref{sec:hilb} for easier reference.

In the final Section~\ref{sec:prime}, we consider the case when the characteristic of the field is a prime. 
In this case, a well-known theorem of 
Zassenhaus states that $Z(\g)$ is a normal domain and our main result states that 
the generators $z_2,\ldots,z_n$ 
in Theorem~\ref{th:th1} generate $Z_n$ over the $p$-center $Z_p(\g)$ up to normal closure.  
 
\begin{theorem}\label{th:th4}
    Suppose that $\F$ has characteristic $p$ and also that $p\geq n+1$. Set $\g=\g(n+2)$.  
    Then $Z(\g)$ is the integral closure of $Z_p(\g)[z_2,\ldots,z_{n}]$ in its field of fractions, and it 
    also coincides with the integral closure of  $Z_p(\g)[z_2,\ldots,z_{n}]$ in the localization
    $Z_p(\g)[y_0^{-1},z_2,\ldots,z_{n}]$. In particular, $Z(\g)\subseteq Z_p(\g)[y_0^{-1},z_2,\ldots,z_{n}]$. 
\end{theorem}

The first author was financially 
supported by a PhD scholarship awarded by CNPq (Brazil). The second author acknowledges the financial support of  the CNPq projects 
\textit{Produtividade em Pesquisa} (project no.: 308212/2019-3)  
and \textit{Universal} (project no.: 421624/2018-3 and 402934/2021-0) and the FAPEMIG project 
\textit{Universal} (project no.: APQ-00818-23). We thank Lucas Calixto for several useful comments on the earlier drafts, to Dmitry Shcheglov and to Csaba Noszály for help with Theorem~\ref{th:th3}. We appreciate the useful observations of the referee.

\section{The standard filiform Lie algebra}
\label{sec:stfil}

The adjoint action of the Lie algebra $\g=\g(n+2)$ on itself can 
be extended to the polynomial ring $\F[\g]=\F[x,y_0,y_1,\ldots,y_n]$, and we define the invariant ring $\F[\g]^\g$ as  
\[
    \F[\g]^\g=\{f\in\F[\g]\mid u(f)=0\mbox{ for all }u\in\g\}.
\]
The commutative algebra $\F[\g]^\g$ is referred to as the {\em algebra of polynomial invariants} of~$\g$. 

Let $\F[\bfy]$ denote the polynomial algebra $\F[y_0,y_1\ldots]$ in infinitely 
many variables and let $\F[\bfy_n]=\F[y_0,\ldots,y_n]$. 
Then the one-dimensional Lie algebra $\n=\left<x\right>$  acts on 
the vector space $V=\left<y_0,y_1,\ldots\right>$ by mapping $y_i$ to $y_{i-1}$ for 
$i\geq 1$ and $y_0$ to $0$. Set $V_n=\left<y_0,\ldots,y_n\right>$ and note 
that $V_n$ is an $\n$-submodule of $V$.  
The Lie algebra $\g(n+2)$ can also be viewed as the semidirect product $V_n\rtimes\n$.

The $\n$-action on $V$ can be extended to $\F[\bfy]$ by 
the Leibniz rule and $\F[\bfy_n]$ is $\n$-invariant for all $n\geq 1$. 
We denote by 
$x(f)$ the image of $f\in\F[\bfy]$ under this action.
Simple computation shows that
\begin{equation}\label{eq:action}
    x(f)=\sum_{i\geq 0}y_i\frac{\partial f}{\partial y_{i+1}}.
\end{equation}
The algebra of $x$-invariants in $\F[\bfy]$ is denoted by $\F[\bfy]^x$; more precisely,
\[
    \F[\bfy]^x=\{f\in\F[\bfy]\mid x(f)=0\}.
\]
The operator $f\mapsto x(f)$ on $\F[\bfy]$ is often referred to as the \emph{down operator}; 
see~\cite{freud}.
The algebra of $x$-invariants can also be described as the 
set of polynomial solutions of the partial differential equation
\begin{equation}\label{eq:pde}
    \sum_{i\geq 0}y_i\frac{\partial f}{\partial y_{i+1}}=0.
\end{equation}

For a Lie algebra $\g$, let $Z(\g)$ denote the center of the universal enveloping algebra $U(\g)$.
Since $y_0,\ldots,y_n$ commute in $\g(n+2)$, the Poincaré--Birkhoff--Witt Theorem gives 
an embedding of the polynomial algebra 
 $\F[\bfy_n]$ into $U(\g(n+2))$. Further, if 
$\F$ has prime characteristic $p$ and $p\geq n+1$, then 
$x^p\in Z(\g(n+2))$ and $\F[x^p,y_0,\ldots,y_n]$ is a subalgebra of $U(\g(n+2))$.   
Note that the action of $x$ on $\F[\bfy_n]$ 
is the same whether it is considered inside $U(\g(n+2))$ or it is considered as a stand-alone 
polynomial algebra under the action defined in~\eqref{eq:action}.

\begin{lemma}\label{lem:isom}
    Let $\g=\g(n+2)$ be a standard filiform Lie algebra. If $\F$ has characteristic zero, then 
\[
    Z(\g)\subseteq\F[\bfy_n]
\]
and 
\[ 
    Z(\g)=\F[\g]^\g=\F[\bfy_n]^x.
\]
If the characteristic of $\F$ is a prime $p$, then
\[
    Z(\g)\subseteq \F[x^p,y_0,\ldots,y_n]
\] 
and, if in addition $p\geq n+1$, then
\[
    Z(\g)=\F[\g]^\g=\F[x^p,y_0,\ldots,y_n]^x.
\]
\end{lemma}
\begin{proof}
Easy induction shows, in $U(\g)$, that  
\[
    [x^k,y_n]=x^ky_n-y_nx^k=kx^{k-1}y_{n-1}+h_k
\]
where $h_k$ is a linear combination of monomials whose degree in $x$ is smaller than $k-1$. 
Suppose that $z\in Z(\g)$ and let $m=x^\alpha y_0^{\alpha_0}\cdots y_n^{\alpha_n}$ be the 
leading term of $z$ in the lexicographic monomial order. Then the last displayed equation 
implies that the leading term of $[z,y_n]$ is 
\[
    \alpha x^{\alpha-1}y_0^{\alpha_0}\cdots y_{n-1}^{\alpha_{n-1}+1}y_n^{\alpha_n}.
\]
But, as $z$ is central, $[z,y_n]=0$, and so either 
$\alpha=0$ or the characteristic of $\F$ is $p$ and $p\mid \alpha$. 
Thus, the inclusions for $Z(\g)$ in $\F[\bfy_n]$ and, in the case of characteristic $p$, 
in $\F[x^p,y_0,\ldots,y_n]$ is verified. One can similarly verify the same containment 
$\F[\g]^\g\subseteq \F[\bfy_n]$ and $\F[\g]^\g\subseteq \F[x^p,y_0,\ldots,y_n]$. 
For the equations regarding $Z(\g)$, note that 
if $\F$ has prime characteristic $p$ and $p\geq n+1$, then $x^p\in Z(\g)$. 
The rest of the lemma follows from the observation above that $\F[\bfy_n]$ and 
$\F[x^p,y_0,\ldots,y_n]$ are embedded into $U(\g)$ under the Poincaré--Birkhoff--Witt Theorem and 
the $x$-action on these subalgebras is the same as in~\eqref{eq:action}.
\end{proof}

From now on, we set $Z_n=Z(\g_{n+2})$ considered as a subalgebra of $\F[\bfy_n]$. Note that the 
$x$-action on $\F[\bfy_n]$ preserves the grading, and hence  $Z_n$ 
is a graded subalgebra.

\begin{example}\label{ex:z1}
    Suppose that the characteristic of $\F$ is zero.
    The ring $Z_1$ is the center of $U(\g(3))$ which coincides with the invariant algebra of 
    $\g(3)$. Note that $\g(3)$ is the three-dimensional Heisenberg Lie algebra. By 
    Lemma~\ref{lem:isom}, 
    \[
        Z_1=\left\{f\in\F[y_0,y_1]\mid y_0\frac{\partial f}{\partial y_1}=0\right\}=
        \left\{f\in\F[y_0,y_1]\mid \frac{\partial f}{\partial y_1}=0\right\}=
        \F[y_0].
    \]
\end{example}

\section{The elements of $\F[\bfy_n]^x$}\label{sec:gens}

In this section, we prove some general facts concerning the elements of  $\F[\bfy_n]^x$ and of the fraction 
field $\fracfield{\F[\bfy_n]^x}$. We will use these facts to determine explicit generators 
in Section~\ref{sec:gens}. In this section $\F$ is a field of characteristic zero.

\begin{lemma}\label{lem:gi}
    Suppose that $z\in\F[\bfy_n]$ and write 
    \[
        z=\sum_{i=0}^k y_n^ig_i
    \]
    where $g_i\in\F[\bfy_{n-1}]$. Then $z\in \F[\bfy_n]^x$ if and only if  
    \[
        x(g_k)=0\quad\mbox{and}\quad x(g_i)=-(i+1)y_{n-1}g_{i+1}\quad\mbox{for all}\quad i\leq k-1.
    \]
\end{lemma}
\begin{proof}
    First, we compute 
\begin{align*}
    x(z)&=x\left(\sum_{i=0}^k y_n^ig_i\right)=\sum_{i=1}^k iy_n^{i-1}y_{n-1}g_i+\sum_{i=0}^k y_n^i\cdot x(g_i)\\&=
    y_n^k\cdot x(g_k)+\sum_{i=0}^{k-1}y_n^i\left((i+1)y_{n-1}g_{i+1}+x(g_i)\right).
\end{align*}
Since $g_i,x(g_i)\in\F[\bfy_{n-1}]$ for all $i\in\{0,\ldots,k\}$, we obtain that the coefficient of 
$y_n^k$ in $x(z)$ is $x(g_k)$ and, for $i\in\{0,\ldots,k-1\}$, the coefficient of $y_n^i$ is $(i+1)y_{n-1}g_{i+1}+x(g_i)$. 
Thus, $x(z)=0$ if and only if $x(g_k)=0$ and $(i+1)y_{n-1}g_{i+1}+x(g_i)=0$ for $i\in\{0,\ldots,k-1\}$. 
\end{proof}

\begin{corollary}\label{cor:gi}
    If $z\in\F[\bfy_n]^x$ is written as in the previous lemma, then
    $z$ is determined by the term $g_0\in\F[\bfy_{n-1}]$.  
\end{corollary}

In the following lemma, we consider monomials ordered lexicographically considering
 the exponents of $y_n,\ldots, y_1,y_0$ in this order.

\begin{lemma}
    Suppose that $z\in\F[\bfy_n]^x$ and write 
    \[
        z=\sum_{i=0}^k y_n^ig_i
    \]
    where $g_i\in\F[\bfy_{n-1}]$ for all $i\in\{0,\ldots,k\}$. Then 
    \begin{enumerate}
        \item $g_k\in\F[\bfy_{n-1}]^x$;
        \item the  leading monomial of $z$ does not contain $y_1$. 
    \end{enumerate}
\end{lemma}
\begin{proof}
    Statement~(1) follows from Lemma~\ref{lem:gi}. Let us prove (2) by induction on $n$. If $n=1$, then 
    the statement is clear, since $Z_1=\F[\bfy_1]^x=\F[y_0,y_1]^x=\F[y_0]$ (see Example~\ref{ex:z1}). Suppose that the statement 
    holds for $n-1\geq 1$ and consider $z\in\F[\bfy_n]$. If $z$ does not contain the variable $y_n$, 
    then $z\in\F[\bfy_{n-1}]^x$ and we are done by the induction hypothesis. Otherwise, write $z=\sum_{i\leq k}y_n^ig_i$. 
    Then the  leading monomial of $z$ is $y_n^km$ where $m$ is the 
    leading monomial of $g_k$. By Statement~(1), $g_k\in\F[\bfy_{n-1}]^x$, and so the induction hypothesis 
    implies that $m$ does not contain $y_1$. Thus,
    $y_n^km$ does not contain $y_1$ also.  
\end{proof}

\begin{theorem}\label{th:Zgen}
    Suppose that $Z\subseteq\F[\bfy_n]$ is a graded subalgebra (with respect to the grading by degree) such that 
    if $z\in Z$ then the  leading monomial of $z$ is not divisible by $y_1$.  
    Suppose that $z_1,\ldots,z_n\in Z$ are homogeneous elements such that $z_1=y_0$ and, for each $i\in\{2,\ldots,n\}$,
\[
    z_i=y_iy_0^{k_i}+z'_i\mbox{ with } k_i\geq 0\mbox{ and } z_i'\in\F[\bfy_{i-1}].
\]
Then the following are valid.
\begin{enumerate}
    \item $Z\subseteq \F[y_0^{-1},z_1,\ldots,z_n]$.
    \item The fraction field of $Z$ is generated by $z_1,\ldots,z_n$.
    \item The elements $z_1,\ldots,z_n$ are algebraically independent over $\F$.
    \item $\dim Z=n$ (Krull dimension).
\end{enumerate} 
\end{theorem}
\begin{proof}
(1) Suppose that $z\in Z$; we need to show that $z$ can be written as a polynomial expression 
in $y_0^{-1}$ and in the given $z_i$. 
We show this by induction on the leading monomial of $z$. The base case of the induction is when 
$z=1$ and the statement in this case is obvious.
Suppose now that $z$ is a homogeneous element of $Z$ and its leading monomial 
is 
\[
    m=y_n^{\alpha_n}\cdots y_2^{\alpha_2}y_0^{\alpha_0}
\]
with $m\neq 1$. Note, for $i\geq 2$, that the leading monomial of $z_i$ is $y_iy_0^{k_i}$. Consider the element 
\[
  z'=z-z_2^{\alpha_2}\cdots z_n^{\alpha_n}y_0^{\alpha_0-\alpha_2k_2-\cdots-\alpha_kk_n}  
\]
Then the leading monomial of $z'$ is smaller than $m$ in the monomial ordering and $z'\in Z$. By the induction 
hypothesis, $z'\in\F[y_0^{-1},z_1,\ldots,z_n]$ which implies that $z\in \F[y_0^{-1},z_1,\ldots,z_n]$.

(2) follows from (1). To prove~(3), take a polynomial $g\in\F[t_1,\ldots,t_n]$ such that  $g(z_1,\ldots,z_n)=0$. 
If $g\neq 0$, then, supposing that the leading term of $g$ is $t_1^{\alpha_1}\cdots t_n^{\alpha_n}$, 
the leading term of $g(z_1,\ldots,z_n)$ is $y_n^{\alpha_n}\ldots y_2^{\alpha_2}y_0^k$ 
for some $k\geq 0$, which is nonzero. Thus,  $g=0$ must hold and $z_1,\ldots,z_n$ are algebraically independent.

(4) follows from (3).
\end{proof}

\section{Some facts concerning the representations of $\slg$}\label{sec:sl2}

Suppose in this section that $\F$ is a field of characteristic zero. It is well known that the Lie algebra $\slg=\left<e,f,h\right>$, where $h=[e,f]$,
acts on the polynomial ring $\F[x_1,x_2]$ by extending the action
\begin{align*}
    e&: x_1\mapsto 0,\quad x_2\mapsto x_1,\\
    f&: x_1\mapsto x_2,\quad x_2\mapsto 0,\\
    h&: x_1\mapsto x_1,\quad x_2\mapsto -x_2\\
\end{align*}
using the Leibniz rule.  This action on $\F[x_1,x_2]$ can be described, for $g\in\F[x_1,x_2]$, as 
\[
    e(g)=x_1\frac{\partial g}{\partial x_2}\qquad\mbox{and}\qquad f(g)=x_2\frac{\partial g}{\partial x_1}.
\]  
Let $U_n$ denote the space of homogeneous polynomials in $\F[x_1,x_2]$ of degree $n$ with basis 
$x_1^n,x_1^{n-1}x_2,\ldots,x_1x_2^{n-1},x_2^n$. Then the action of $e$ and $f$ can be described as 
\[
    e(x_1^ix_2^{n-i})=(n-i)x_1^{i+1}x_2^{n-i-1}\quad\mbox{and}\quad f(x_1^ix_2^{n-i})=ix_1^{i-1}x_2^{n-i+1}.
\]
The basis $x_1^n,x_1^{n-1}x_2,\ldots,x_2^n$ consists of eigenvectors for the operator $h$ with eigenvalues 
$n,n-2,\ldots,-n+2,-n$, respectively. The eigenvectors of $h$ in a representation of $\slg$ are often referred 
to as \emph{weight vectors} and the corresponding eigenvalues as \emph{weights}.
The following result is well-known.

\begin{theorem}\label{th:slmods}
    For $n\geq 0$, the space $U_n$ is irreducible considered as an $\slg$-module. Furthermore,  each finite-dimensional  irreducible $\slg$-module  is isomorphic to $U_n$ for some $n\geq 0$.   
\end{theorem}

Note that the $\slg$-module $U_n$ can also be recognized by the weights 
(that is, the $h$-eigenvalues)
which are $n,n-2,\ldots,-n+2,-n$. 
The reason why we are interested in the $\slg$-modules $U_n$ is because the action of $e$ on $U_n$ is 
isomorphic to the action of $x$ on $V_n=\left<y_0,\ldots,y_n\right>$ defined in Section~\ref{sec:stfil} 
in the context of standard filiform Lie algebras.

\begin{lemma}\label{lem:xe}
    Suppose that $n\geq 0$ and, for $i\in\{0,\ldots,n\}$, let $w_i=(1/i!)x_1^{n-i}x_2^i$. Then 
    $e(w_0)=0$ and $e(w_i)=w_{i-1}$ for $i\in\{1,\ldots,n\}$. 
\end{lemma}
\begin{proof}
    This is easy: just calculate for $i\geq 0$ that 
    \[
        e(w_{i})=\frac{1}{i!}e(x_1^{n-i}x_2^i)=i\cdot\frac 1{i!}x_1^{n-i+1}x_2^{i-1}=\frac 1{(i-1)!}x_1^{n-i+1}x_2^{i-1}
        =w_{i-1}.
    \] 
\end{proof}

For a Lie algebra $\g$ and for a $\g$-module $U$, let $U^\g$ denote the space 
\[
    U^\g=\{u\in U\mid a(u)=0\mbox{ for all } a\in \g\}
\] 
of $\g$-invariants in $U$. 
If $\g=\left<x\right>$, then we write $U^x$ for $U^\g$ and $U^x$ is the zero eigenspace of $x$ acting 
on $U$. 

\begin{corollary}\label{cor:semiinv}
    The following are valid for a  finite-dimensional $\slg$-module $U$ written as 
    a direct sum $U=W_1\oplus\cdots\oplus W_d$ of simple $\slg$-submodules. 
    \begin{enumerate}
    \item  For $i\in\{1,\ldots,d\}$, let $z_i\in W_i^e\setminus\{0\}$ 
    (which is unique up scalar multiple). Then $\{z_1,\ldots,z_d\}$ is a basis of $U^e$. 
    \item $\dim U^e=\dim U^f=d$. 
    \item $\dim U^e=\dim Y_0+\dim Y_1$ where $Y_0$ and $Y_1$ are the eigenspaces in $U$ of the operator $h$ corresponding to the eigenvalues $0$ and $1$, respectively.
    \end{enumerate}
\end{corollary}
\begin{proof}
    By Weyl's Theorem (see~\cite[Theorem~8, Section~III.7]{jac}),  
    \begin{equation}\label{eq:ws}
        U=W_1\oplus\cdots \oplus W_d
    \end{equation} where the $W_i$ are irreducible $\slg$-modules. Then each 
    $W_i$ is isomorphic to $U_{k_i}$ where $k_i\geq 0$ (Theorem~\ref{th:slmods}). By Lemma~\ref{lem:xe}, the matrix of $e$ on $W_i$ is conjugate to a single Jordan block that corresponds to the eigenvalue zero. Thus, the matrix of $e$ on $U$ is conjugate to a block-diagonal matrix with $d$ Jordan blocks each of which corresponds to the eigenvalue zero. Noting that $U^e$ is the zero eigenspace of $e$ considered 
    as an operator on $U$, 
    we obtain assertions (1) and (2).  
    For the assertion that $d=\dim Y_0+\dim Y_1$, choose for each $W_i$ a basis consisting of 
    $h$-eigenvectors and note that each irreducible component $W_i$ contributes with 
    dimension one to either $Y_0$ (if $k_i$ is even) or to $Y_1$ (if $k_i$ is odd).
\end{proof}

\begin{lemma}\label{lem:ugen}
    Suppose that $U$ is a finite-dimensional $\slg$-module. Suppose that $z\in U^e$ is an $h$-eigenvector of eigenvalue $w$. Then $z$ is contained in $\bigoplus_i (W_i)^e$ where the $W_i$ are the 
    irreducible components of $U$ that are isomorphic to $U_w$. Furthermore, the $\slg$-submodule 
    generated by $z$ is irreducible and is isomorphic to $U_w$. 
\end{lemma}
\begin{proof}
The fact that $z\in\bigoplus_i (W_i)^e$ follows from the weight space decomposition~\eqref{eq:ws} for $U$.
One can show by induction that $f^i(z)$ is an $h$-eigenvector with eigenvalue 
$w-2i$ for all $i\in \{0,\ldots,n\}$ and these vectors form the basis of an $\slg$-submodule isomorphic 
to $U_w$. 
\end{proof}

\begin{theorem}\label{th:cg}
    Suppose that $U_m$ and $U_n$ are $\slg$-modules as defined above. Then the following decomposition 
    of $\slg$-modules is valid:
    \[
        U_m\otimes U_n\cong\bigoplus_{i=0}^{\min\{m,n\}}U_{m+n-2i}.
    \]
Suppose that $u_0,\ldots,u_m$ and $w_0,\ldots,w_n$ are bases for $U_m$ and $U_n$ consisting of $h$-eigenvectors, respectively,
such that $e(u_0)=e(w_0)=0$ and $e(u_i)=u_{i-1}$ and $e(w_j)=w_{j-1}$ for $i\in\{1,\ldots,m\}$
and $j\in\{1,\ldots,n\}$. Define, for $i=0,1,\ldots,\min\{m,n\}$,
\begin{equation}\label{eq:zi}
   z_i= \sum_{d=0}^i(-1)^d {u_d\otimes w_{i-d}}.
\end{equation}
Then $e(z_i)=0$ and the irreducible component  of $U_m\otimes U_n$ isomorphic to $U_{m+n-2i}$ is generated by 
$z_{i}$.
\end{theorem}
\begin{proof}
    The decomposition of $U_m\otimes U_n$ is well known as the Clebsch--Gordan formula 
    for $\slg$~\cite[Theorem~2.6.3]{kowalski}.
    The fact that $e(z_i)=0$ is true since 
\begin{align*}
    e(z_i)&=e\left (\sum_{d=0}^i(-1)^d{u_d\otimes w_{i-d}}\right )=
    \sum_{d=0}^i(-1)^d{e(u_d)\otimes w_{i-d}}+\sum_{d=0}^i(-1)^d{u_d\otimes e(w_{i-d})}\\&=
    \sum_{d=1}^i(-1)^d{u_{d-1}\otimes w_{i-d}}+\sum_{d=0}^{i-1}(-1)^d{u_d\otimes w_{i-d-1}}=0.
\end{align*}     
    Further, it also follows that 
    $z_i$ is an $h$-eigenvector of eigenvalue $n+m-2i$. By Lemma~\ref{lem:ugen}, $z_i$ 
    is contained in the direct sum of the $\slg$-submodules that are isomorphic to 
    $U_{n+m-2i}$, but the Clebsch--Gordan decomposition implies that there is a unique such submodule and 
    this must be generated by $z_i$. 
\end{proof}

The element $z_i$ defined in~\eqref{eq:zi} can be viewed as a Casimir element for the operator $e$
on the tensor product $U_m\otimes U_n$. 

\begin{lemma}\label{lem:ef}
    Let $z\in U_n$ be an $h$-eigenvector  of weight $n-2i$ with some $i\in\{0,\ldots,n\}$. Then 
    \[
        e(f(z))=(n-i)(i+1)z.
    \]
\end{lemma}
\begin{proof}
    We can assume without loss of generality that $U_n=\left<x_1^n,x_1^{n-1}x_2,\ldots,x_1x_2^{n-1},x_2^n\right>$ 
    and $z=x_1^{n-i}x_2^i$. Then 
    \[
        e(f(z))=e(f(x_1^{n-i}x_2^i))=(n-i)(i+1)x_1^{n-i}x_2^i.
    \]
\end{proof}

\begin{corollary}\label{cor:d_prod}
    Let $U_n$ be an $(n+1)$-dimensional irreducible $\slg$-module and consider 
    the polynomial algebra $\F[U_n]$ as an 
    $\slg$-module. For $i\geq 0$, 
    let $U_{n,i}$ denote the space of homogeneous polynomials of $\F[U_n]$ of degree $i$. 
    Let $z_1\in U_{n,k_1}^e$ and $z_2\in U_{n,k_2}^e$ that are also $h$-eigenvectors 
    of weight $w_1$ and $w_2$, respectively. 
    Suppose that $d\leq \min\{w_1,w_2\}$ and define
    for $i\in\{1,2\}$ and for $k\in\{0,\ldots,d\}$, 
    \[
        z_{i,0}=z_i\quad\mbox{and}\quad z_{i,k}=\frac 1{k(w_i-k+1)}f(z_{i,k-1})\quad\mbox{for}\quad k\geq 1.
    \]
    Set   
    \[
        \dprod{z_1}{z_2}d = \sum_{i=0}^{d} (-1)^iz_{1,i}z_{2,d-i}.
    \]
    Then $\dprod{z_1}{z_2}d\in U_{n,k_1+k_2}^e$ and it is an 
    $h$-eigenvector of weight $w_1+w_2-2d$. Furthermore, letting $y\in U_{n,1}^e\setminus\{0\}$ be fixed
    and $z_1,\ldots,z_m$ be a basis of $U_{n,k-1}^e$ formed by $h$-eigenvectors, 
    the elements $\dprod{z_i}{y}{d}$, with $i\in\{1,\ldots,m\}$ and $d\geq 0$ 
    generate $U_{n,k}^e$ and  are $h$-eigenvectors.
\end{corollary}
\begin{proof}
    Suppose that $W_1$ and $W_2$ are the irreducible components of 
    $U_{n,k_1}$ and $U_{n,k_2}$ that contain $z_1$ and $z_2$, respectively (they exist by Lemma~\ref{lem:ugen}). 
    Then $W_1\cong U_{w_1}$ and $W_2\cong U_{w_2}$.
    The fact that $e(z_{i,0})=0$, for $i=1,2$, is clear. Suppose that $k\geq 1$ note that $z_{i,k-1}$ 
    is an $h$-eigenvector of weight $w_i-2(k-1)$. Thus, Lemma~\ref{lem:ef} implies that
    \[
         e(z_{i,k})=\frac{1}{k(w_i-k+1)}e(f(z_{i,k-1}))=\frac{k(w_i-k+1)}{k(w_i-k+1)}z_{i,k-1}=z_{i,k-1}.
    \]
    The product space $W_1W_2$ is an epimorphic image of the tensor product $W_1\otimes W_2$ under the $\slg$-homomorphism 
    $\psi:W_1\otimes W_2\to W_1W_2$ induced by multiplication. Observe that $\dprod{z_1}{z_2}{d}$ is the image of 
\[    
    \sum_{i=0}^{d} (-1)^iz_{1,i}\otimes z_{2,d-i}
\] 
under $\psi$, and so Theorem~\ref{th:cg} implies that $e(\dprod{z_1}{z_2}d)=0$ and also that 
$(W_1W_2)^e$ is generated by such elements. The final assertion follows from the fact that $U_{n,k}=U_{n,k-1}U_{n,1}$ and that $U_{n,1}$ is irreducible, and in particular $U_{n,1}^e=\left<y\right>$.   
\end{proof}

\section{Explicit generators of $Z_n$}\label{sec:expgens}

In this section we return to our investigation of the invariant algebra 
$Z_n=\F[\bfy_n]^x$ as seen in Sections \ref{sec:stfil}--\ref{sec:gens}. 
Assume throughout this section that $\F$ is a field of characteristic zero. 
Note that $x$ acts on the vector space 
$V_n=\left<y_0,y_1,\ldots,y_n\right>$ and Lemma~\ref{lem:xe} shows that its action is equivalent 
to the action of $e\in\slg$. Thus, there exists $\hat f_n\in\mbox{End}(V_n)$ such that 
the map $x\mapsto e,\ \hat f_n\mapsto f,\ [x,\hat f_n]\mapsto h$ can be extended to an isomorphism of Lie algebras
$\langle x,\hat f_n, [x,\hat f_n]\rangle\to \slg$. In fact, $\hat f_n$ is determined in~\cite[Theorem~3.1]{bed_ukrj} and 
in \cite[Theorem~3.1]{bed_arx} as 
\begin{equation}\label{eq:fhat}
    \hat f_n(y_i)=(i+1)(n-i)y_{i+1}\quad\mbox{for all}\quad i\in\{0,\ldots,n\}.
\end{equation}
It is easy to compute, for all $y_i$, that $[x,\hat f_n](y_i)=(n-2i)y_i$ and that $[x,[x,\hat f_n]]=-2x$,  while
$[\hat f_n,[x,\hat f_n]]=2\hat f_n$ which verifies the isomorphism $\langle x,\hat f_n,h\rangle\to \slg$.
Because of the isomorphism with $\slg$,  $\hat h=[x,\hat f_n]$ is an operator on $V_n$ with eigenvectors 
$y_0,y_1,\ldots,y_n$ corresponding to eigenvalues $n,n-2,\ldots,-n+2,-n$.
For $k\geq 0$, let $Z_{n,k}$ denote the degree-$k$ homogeneous component of $Z_n$. 
The algebras $\F[\bfy_n]$ and $\F[U_n]$ in Corollary~\ref{cor:d_prod} are isomorphic as 
$\slg$-modules. Furthermore, under the obvious isomorphism between $\F[\bfy_n]$ and $\F[U_n]$, 
$Z_n$ corresponds to $\F[U_n]^e$, while $Z_{n,k}$ corresponds to $U_{n,k}^e$. 
Thus, the following result is an immediate consequence of Corollary~\ref{cor:d_prod}.

\begin{proposition}\label{prop:dprod}
    Suppose that  $z_1\in Z_{n,k_1}$ and 
    $z_2\in Z_{n,k_2}$ 
    are $\hat h$-eigenvectors with eigenvalues $w_1$ and $w_2$ respectively. 
    Suppose that $d\leq \min\{w_1,w_2\}$ and define
    for $i\in\{1,2\}$ and for $k\in\{0,\ldots,d\}$, 
    \[
        z_{i,0}=z_i\quad\mbox{and}\quad z_{i,k}=\frac 1{k(w_i-k+1)}\hat f_n(z_{i,k-1})\quad\mbox{for}\quad k\geq 1.
    \]
    Define   
    \begin{equation}\label{eq:zi_pol}
        \dprod{z_1}{z_2}d = \sum_{i=0}^{d} (-1)^iz_{1,i}z_{2,d-i}.
    \end{equation}
    Then $\dprod{z_1}{z_2}d\in Z_{n,k_1+k_2}$ and this element is an $\hat h$-eigenvector of weight $w_1+w_2-2d$. 
    Furthermore, letting $z_1,\ldots,z_m$ be a basis of $Z_{n,k-1}$ formed by $\hat h$-eigenvectors, 
    the elements $\dprod{z_i}{y_0}{d}$, with $i\in\{1,\ldots,m\}$ and $d\geq 0$ 
    generate $Z_{n,k}$ and  are $\hat h$-eigenvectors.
\end{proposition}

The element $\dprod zz{d}$ appears as $\tau_d(z)$ in~\cite{bed_arx,bed_ukrj}. 
Note that using equation~\eqref{eq:zi_pol}, we obtain that 
\[
    \dprod{y_0}{y_0}{d}=\sum_{i=0}^d(-1)^i y_iy_{d-i}.
\]
Applying Proposition~\ref{prop:dprod} recursively for the homogeneous components $Z_{n,1}$, $Z_{n,2}$, etc., 
we obtain the following. 
\begin{corollary}
    For $k\geq 1$, the homogeneous component 
    $Z_{n,k}$ is generated as a vector space by elements of the form 
    \[
        \dprod{(\dprod{(\dprod{y_0}{y_0}{d_1})}{y_0)}{d_2}\cdots}{y_0}{d_{k-1}}
    \]
    where the number of factors is $k$ and the $d_i$ are chosen in such a way that 
    $\circ_{d_i}$ be defined for each $i$. In particular, the ring $Z_n$ admits a minimal generating set 
    formed by polynomials of this form.
\end{corollary}

In some cases, it is possible to obtain explicit generators for the invariant ring $\F[\bfy_n]^x$ as shown 
in the following example. 

\begin{example}\label{ex:z14}
    As noted in Example~\ref{ex:z1}, $Z_1=\F[y_0]$, and it is well-known that  
    \[
        Z_2=\F[y_0,\dprod{y_0}{y_0}{2}]=\F[y_1,2y_0y_2-y_1^2]
    \]
    (see~\cite{dixmier58}).
The generators for $Z_3$ were determined by Dixmier~\cite{dixmier58}, and we can write them as 
\[
        Z_3=\F[z_1,z_2,z_3,\zeta_4]
\] 
where 
\begin{align*}
    z_1 &= y_0;\\
    z_2 &= 2y_0y_2-y_1^2=y_0\circ_2 y_0;\\
    z_3 &= 3y_0^2 y_3 - 3y_0 y_1 y_2 + y_1^3=-y_0\circ_2 y_0\circ_1 y_0;\\
    \zeta_4&=\frac{z_2^3 + z_3^2}{y_0^2}=-3y_1^2y_2^2 + 8y_0y_2^3 + 6y_1^3y_3 - 18y_0y_1y_2y_3 + 9y_0^2y_3^2\\
    &=(-3/2)y_0\circ_2 y_0\circ_1 y_0\circ_3 y_0.
\end{align*}
The expressions in $\circ_d$ are left normed; 
that is, $a\circ_{d_1}b\circ_{d_2}c=(a\circ_{d_1}b)\circ_{d_2}c$. 
According to~\cite{Ooms2} The algebra $\F[\bfy_4]^x$ is generated by 
\begin{align*}
    z_1 &= y_0;\\
    z_2 &= 2y_0y_2-y_1^2=y_0\circ_2 y_0;\\
    z_3 &= 3y_0^2 y_3 - 3y_0 y_1 y_2 + y_1^3=-y_0\circ_2 y_0\circ_1 y_0;\\
    z_4 &= 2y_0y_4-2y_1y_3+y_2^2=y_0\circ_4 y_0;\\
    \zeta_5&=\frac{z_2^3+z_3^2-3y_0^2z_2z_4}{y_0^3}=2y_2^3 - 6y_1y_2y_3 + 9y_0y_3^2 + 6y_1^2y_4 - 12y_0y_2y_4\\
    &=-2y_0\circ_2 y_0\circ_4 y_0.
\end{align*}
\end{example} 

Since $\F[\bfy_n]^x$ is a noetherian ring, by the Mauer--Weitzenb\"ock Theorem~\cite[Theorem~6.1]{freudbook}, one may want to describe its Noether normalization. The algebra $\F[\bfy_2]^x$ is freely generated by $z_1$ and $z_2$ and so 
there is nothing to do. On the other hand, $\F[\bfy_3]^x$ is an integral extension of $\F[z_1,z_2,\zeta_4]$, 
or of $\F[z_1,z_3,\zeta_4]$, but not of $\F[z_1,z_2,z_3]$. This is because $z_3$ satisfies the 
integral equation $t^2+z_2^3-y_0^2\zeta_4=0$. Similarly, we obtain that $\F[\bfy_4]$ is integral over 
$\F[z_1,z_2,z_4,\zeta_5]$ or over $\F[z_1,z_3,z_4,\zeta_5]$, but not over $\F[z_1,z_2,z_3,z_4]$.  

\medskip

\noindent {\bf Problem.} Determine generators for the Noether normalization of $\F[\bfy_n]^x$ for $n\geq 5$.

\medskip

The algebra $\F[\bfy_5]$ is known to be minimally generated by 23 generators that satisfy 168 relations; see the discussion after Example~27 in \cite{Ooms1}. Using the computational algebra system Magma~\cite{Magma}, 
we determined explicit expressions in terms of the $\circ_d$ operation for $Z_5$, $Z_6$ and $Z_8$. See  Appendix~\ref{app:gens}.

\subsection{The first generator sequence}\label{sec:firstsec}
We use Proposition~\ref{cor:d_prod} to obtain an explicit generating set for the 
field of fractions of $\F[\bfy_n]^x$. Set $z_1=y_0$, and, for $i\geq 2$, define 
\[
    z_i=\left\{\begin{array}{ll} \dprod{y_0}{y_0}{i}&\mbox{if $i$ is even}\\
                    \dprod{y_0}{(\dprod{y_0}{y_0}{i-1})}1&\mbox{if $i$ is odd.}\end{array}\right.
\]

The explicit expressions for the first couple of values of the $z_i$ are 
as follows:
\begin{align*}
    z_1&=y_0;\\
    z_2&=2y_0y_2 - y_1^2;\\
    z_3&=3y_0^2y_3 - 3y_0y_1y_2 + y_1^3;\\
    z_4&=2y_0y_4 - 2y_1y_3 + y_2^2;\\
    z_5&=5y_0^2y_5 - 5y_0y_1y_4 + y_0y_2y_3 + 2y_1^2y_3 -y_1y_2^2;\\
    z_6&=2y_0y_6 - 2y_1y_5 + 2y_2y_4 - y_3^2;\\
    z_7&=7y_0^2y_7 - 7y_0y_1y_6 + 3y_0y_2y_5 - y_0y_3y_4 + 2y_1^2y_5 - 
    2y_1y_2y_4 + y_1y_3^2.
\end{align*}

Recall that the lexicographic monomial order on $\F[\bfy_n]$ is taken ordering the variables
in increasing order as $y_n,y_{n-1},\ldots,y_0$. 

\begin{lemma}\label{lem:leadtermz}
    The element $z_i\in \F[\bfy_n]^x$ for all $i\leq n$. Furthermore, the leading term of $z_i$ in the lexicographic monomial order is $2y_iy_0$ if $i$ is even, while the leading term is $i\cdot y_iy_0^2$ when $i$ is odd.  
\end{lemma}
\begin{proof}
    The fact that $z_i\in\F[\bfy_n]^x$ follows from Corollary~\ref{cor:d_prod}. 
    Let us show the statement concerning the leading term. For $i\geq 2$ even, we have that 
    \[
        z_i=\sum_{j=0}^i(-1)^jy_jy_{i-j}.
    \]
    The leading term of $z_i$ is the term that contains $y_i$ which is $2y_0y_i$ as claimed. 
   Now for $i$ odd, the definition of $\circ_d$ in~\eqref{eq:zi_pol} and the definition 
   of $\hat f$ in~\eqref{eq:fhat} imply that
    \begin{align*}
        z_i&=\frac 1{2n-2i+2} y_0\hat f_n\left(\sum_{j=0}^{i-1}(-1)^jy_jy_{i-1-j}\right)-y_1\sum_{j=0}^{i-1}(-1)^jy_jy_{i-1-j}\\
        &= \frac 1{2n-2i-2}y_0\left(\sum_{j=0}^{i-1}(-1)^j(j+1)(n-j)y_{j+1}y_{i-1-j}+\right. \\&\left.\sum_{j=0}^{i-1}(-1)^j(i-j)(n-i+j+1)y_jy_{i-j}\right)\\
        &-y_1\sum_{j=0}^{i-1}(-1)^jy_jy_{i-1-j}.
    \end{align*}
    The leading term of this expression is the only term containing $y_i$, which is $iy_0^2y_i$ as claimed.
\end{proof}

\subsection{The second sequence}
We define another sequence of elements in $\F[\bfy]^x$ which appears in~\cite{sw}. (Recall we assume that the 
characteristic of $\F$ is zero.)
Set, $w_1=y_0$ and for $n\geq 1$, 
\begin{equation}\label{eq:w}
    w_n=\frac{(-1)^n}{n!}y_1^n+\sum_{i=0}^{n-1}\frac{(-1)^i}{j!}y_0^{n-1-i}y_1^iy_{n-i}.
\end{equation}
The first couple of values of the sequence $w_i$ are as follows.
{\allowdisplaybreaks
\begin{align*}
    w_1 & = y_0\\ 
    w_2 & = y_0y_2 - \frac 12 y_1^2\\
    w_3 & = y_0^2y_3 - y_0y_1y_2 + \frac 13 y_1^3\\
    w_4 & = y_0^3y_4 - y_0^2y_1y_3 + \frac 12 y_0y_1^2y_2 - \frac 18 y_1^4\\
    w_5 & = y_0^4y_5 - y_0^3y_1y_4 + \frac 12 y_0^2y_1^2y_3 - \frac 16 y_0y_1^3y_2 + \frac 1{30} y_1^5\\
    w_6 & = y_0^5y_6 - y_0^4y_1y_5 + \frac 12 y_0^3y_1^2y_4 - \frac 16 y_0^2y_1^3y_3 + \frac 1{24} y_0y_1^4y_2 - \frac 1{144}y_1^6\\
    w_7 & = y_0^6y_7 - y_0^5y_1y_6 + \frac 12 y_0^4y_1^2y_5 - \frac 16 y_0^3y_1^3y_4 + \frac 1{24}y_0^2y_1^4y_3 - \frac 1{120}y_0y_1^5y_2 + 
    \frac 1{840}y_1^7\\
    w_8 & = y_0^7y_8 - y_0^6y_1y_7 + \frac 12 y_0^5y_1^2y_6 - \frac 16 y_0^4y_1^3y_5 + \frac 1{24}y_0^3y_1^4y_4 - \frac 1{120}y_0^2y_1^5y_3 + 
    \frac 1{720}y_0y_1^6y_2 - \frac 1{5760}y_1^8.
\end{align*}}

\begin{lemma}\label{lem:wi}
    We have that $w_n\in\F[\bfy]^x$ for all $m\geq 1$. Furthermore, the leading term of 
    $w_n$ is $y_0^{i-1}y_n$. 
\end{lemma}
\begin{proof}
    The claim that $w_n\in\F[\bfy]^x$ follows from \cite[Lemma~3.1]{sw}, while the claim concerning the 
    leading term is easily verified by inspection of~\eqref{eq:w}.
\end{proof}

\begin{theorem}\label{th:th11}
    If the characteristic of $\F$ is zero, then the following are valid.
    \begin{enumerate}
        \item The sets  $\{z_1,z_2,\ldots,\}$ and $\{w_1,w_2,\ldots\}$ are algebraically 
        independent in $\F[\bfy]$.
        \item The fraction field of the algebra $Z_n$ is generated by 
    $z_1,\ldots,z_n$ and also by $w_1,\ldots,w_n$. 
    \item $Z_n$ lies in $\F[y_0^{-1},z_1,z_2,\ldots,z_n]$ and also in $\F[y_0^{-1},w_1,w_2,\ldots,w_n]$.
    \item The Krull dimension of $Z_n$ is $n$.
    \end{enumerate} 
\end{theorem}
\begin{proof}
    This follows from Theorem~\ref{th:Zgen}, Lemma~\ref{lem:leadtermz} and from Lemma~\ref{lem:wi}.
\end{proof}

Theorem~\ref{th:th11} implies Theorem~\ref{th:th1}.







\section{The Hilbert series of $Z_n$}\label{sec:hilb}

In this section $\F$ is a field of characteristic zero.
Suppose, as in Section~\ref{sec:stfil}, that $V_n=\left<y_0,\ldots,y_n\right>$ is an  $\F$-vector space of dimension $n+1$ and consider $V_n$ as an $x$-module. 
The action of $x$ on $V_n$ is equivalent to the action of $e\in\slg$ on the $(n+1)$-dimensional irreducible 
$\slg$-module $U_n$ by Lemma~\ref{lem:xe}. Thus, the 
invariant ring $Z_n=\F[\bfy_n]^x$ is isomorphic to the algebra of polynomial invariants $\F[U_n]^e$ 
where $\F[U_n]$ is the polynomial algebra generated by $U_n$.  Furthermore, this isomorphism
is an isomorphism of graded algebras with respect to the grading by degree on $\F[\bfy_n]$ and 
$\F[U_n]$. For $n\geq 1$, and $d\geq 0$, let $\delta_{n,d}$ denote the dimension of the degree-$d$ homogeneous 
component of $Z_n$ (which is equal to the dimension of the degree-$d$ 
homogeneous component of $\F[U_n]^e$). Set 
\[
H_n(t)=\sum_{d=0}^\infty \delta_{n,d}t^d.
\]

Consider the basis of $U_n$ formed by $u_{i}=x_1^{i}x_2^{n-i}$ for $i\in\{0,\ldots,n\}$. 
Then 
\[
    h(u_i)=(-n+2i)u_i 
\]
and each 
monomial $m=u_0^{\alpha_0}\cdots u_n^{\alpha_n}$ in the $u_i$ is an 
eigenvector of $h$ with eigenvalue $w(m)$ (which can be calculated explicitly if needed). 
The eigenvalue $w(m)$ is also 
called the weight of $m$. The degree of such a monomial is $\sum_i\alpha_i$.
Noting that $U_{n,d}$ is isomorphic, as an $\slg$-module, to the 
$d$-th symmetric power $S^d(U_n)$ of the irreducible $\slg$-module $U_n$, 
the following result follows at once from Corollary~\ref{cor:semiinv}.

\begin{theorem}\label{th:deltand}
    $\delta_{n,d}$ is equal to the number of irreducible components of the symmetric power 
    $S^d(U_n)$ considered as an $\slg$-module. Moreover, if either $n$ or $d$  is even then 
    $\delta_{n,d}$ is equal to the number of monomials in the $u_i$ with degree $d$ and weight zero; 
   otherwise $\delta_{n,d}$ is equal to the number of such monomials with degree $d$ and weight one.
\end{theorem}
%
%

\begin{example}\label{ex:n1}
    Consider for example the case of  $n=1$. Then, for each $d\geq 0$,  
    \[
        \delta_{1,d}=1
    \]
    since, for even $d$, the only weight zero monomial of degree $d$ is $u_0^{d/2}u_1^{d/2}$, while for 
    odd $d$, the only weight one monomial of degree $d$ is $u_0^{(d+1)/2}u_1^{(d-1)/2}$. This corresponds to the fact that $Z_1$ is generated by $y_0$ and the homogeneous component of degree 
    $d$ is generated by $y_0^d$ (see Example~\ref{ex:z1}). The Hilbert series is 
\[
        H_1(t)=\sum_{d=0}^\infty t^d=\frac{1}{1-t}.
\]
This shows also that $Z_1$ is generated by $y_0$ which is a generator of degree $1$
and $Z_1$ is isomorphic to the polynomial algebra $\F[t]$ in one variable $t$. 
\end{example}

\begin{example}\label{ex:n2}
    Suppose that $n=2$. Then the weight zero monomials of degree $d$ are 
    $u_0^iu_1^{d-2i}u_2^i$ with $i\in\{0,\ldots \lfloor d/2\rfloor\}$. 
    Thus, $\delta_{2,d}=\lfloor d/2\rfloor +1$ and the sequence $\delta_{2,d}$ is 
    \[
        1,1,2,2,3,3,4,4,5,5,\ldots.
    \]
   Easy computation (distinguishing between 
    $d$ odd and $d$ even) shows that 
    \[
        \delta_{2,d}=\delta_{2,d-1}+\delta_{2,d-2}-\delta_{2,d-3}
    \]
    for $d\geq 3$ with initial values $\delta_{2,0}=\delta_{2,1}=1$ and $\delta_{2,2}=2$.
    Hence, the characteristic polynomial of $\delta_{2,d}$ considered as a recursive sequence is 
    \[ 
        t^3-t^2-t+1=(1-t)(1-t^2).
    \] 
    In fact, the Hilbert series is
    \[
        H_2(t)=\sum_{d=0}^\infty\delta_{2,d}t^d=\frac 1{(1-t)(1-t^2)}
    \]
    reflecting the fact that $Z_2=\F[z_1,z_2]=\F[y_0,2y_0y_2-y_1^2]$ generated by a generator 
    of degree one and a generator of degree~two and these generators are algebraically independent.
\end{example}

For $k,d,n\in\mathbb N_0$, let $p(k,d,n)$ denote the number of partitions of $k$ with $d$ parts each of which of size at most $n$ allowing parts of size zero. 
\begin{theorem}\label{th:part}
    If $n,d\geq 0$, then $\delta_{n,d}=p(\lfloor dn/2\rfloor,d,n)=p(\lfloor dn/2\rfloor,n,d)$. Consequently, 
    $\delta_{n,d}=\delta_{d,n}$ 
\end{theorem}
\begin{proof}
    Suppose that 
    $m=u_0^{\alpha_0}u_1^{\alpha_1}\cdots u_n^{\alpha_n}$ is a monomial in the $u_i$ of degree $d$;
    that is $\sum_{i\geq 0}\alpha_i=d$. 
    Define 
    the function $w'$ on the generators $u_0,\ldots,u_n$ as 
    $w'(u_i)=i$ and extend it multiplicatively to the set of monomials in $u_0,\ldots,u_n$. Then 
    $w'(m)=(w(m)+dn)/2$ where $w(m)$ is the weight (that is, the $h$-eigenvalue) of 
    the monomial $m$. If either $n$ is even or $d$ is even, then 
    $\delta_{n,d}$ is equal to the number of monomials of degree $d$ and weight zero, and if $m$ has degree $d$, then $w(m)=0$ if and only 
    if $w'(m)=dn/2$. If both $n$ and $d$ are odd, then $\delta_{n,d}$ is equal to the number of monomials 
    of degree $d$ and weight minus one (which is equal to the number of monomials of degree $d$ and weight one) and if the degree of $m$ is $d$, then $w(m)=-1$ if and only if $w'(m)=(dn-1)/2=\lfloor dn/2\rfloor$. 
    So in both cases, $\delta_{n,d}$ is equal to the number of monomials $m$ of degree $d$ and 
    $w'(m)=\lfloor dn/2\rfloor$.   
    That is, we need to count monomials $m=u_0^{\alpha_0}u_1^{\alpha_1}\cdots u_n^{\alpha_n}$
    such that 
    \[
        \sum_{i=0}^n\alpha_i=d\quad\mbox{and}\quad \sum_{i=0}^ni\cdot \alpha_i= \lfloor dn/2\rfloor. 
    \]
    Now the sequence $\alpha_0,\alpha_1,\ldots,\alpha_n$ can be identified with the partition 
    of $\lfloor dn/2\rfloor$ in which there are $\alpha_i$ parts of size $i$ for all $i\in\{0,\ldots,n\}$. 
    Omitting the parts of size zero, we obtain a correspondence between the set of such monomials $m$ 
    and  partitions of $\lfloor dn/2\rfloor$ with at most $d$ parts each of which has size at most $n$. 

    Transposing the partitions in question, we obtain that $p(\lfloor dn/2\rfloor,d,n)=p(\lfloor dn/2\rfloor,n,d)$, and this implies the final statement of the theorem.
\end{proof}


\begin{example}\label{ex:n3}
Consider the case $n=3$. Then $\delta_{3,d}$ is equal to the number of partitions of 
$\lfloor 3d/2\rfloor$ with 
at most $d$ parts of size at most $3$. The first elements of this sequence are 
$1$, $1$, $2$, $3$, $5$, $6$, $8$, $10$, $13$, $15$, $18$, $21$, $25$, $28$, $32$, $36$, etc. 
According to~\cite[A001971]{oeis}, this sequence $a(d)=\delta_{3,d}$ satisfies the recurrence relation
\[
a(d)= 2a(d-1) - a(d-2) + a(d-4) - 2a(d-5) + a(d-6).
\]
Furthermore, its generating function is 
\[
    H_3(t)=\frac{1 - t^6}{(1 - t)(1 - t^2)(1 - t^3)(1 - t^4)}.
\]
It is known that $\F[\bfy_3]^x=Z_3$ is generated by $y_0$, $2y_0y_2-y_1^2$, 
$3y_0^2 y_3 - 3y_0 y_1 y_2 + y_1^3$, 
and $-3y_1^2y_2^2 + 8y_0y_2^3 + 6y_1^3y_3 - 18y_0y_1y_2y_3 + 9y_0^2y_3^2$;
that is, a generator of degree one, one of degree 2, one of degree three, and one of degree 4
(see Example~\ref{ex:z14}).
\end{example}

\begin{example}\label{ex:n4}
Consider the case $n=4$. Then $\delta_{4,d}$ is equal to the number of partitions of 
$\lfloor 4d/2\rfloor=2d$ with 
at most $d$ parts of size at most $4$. The first elements of this sequence are 
$1$, $1$, $3$, $5$, $8$, $12$, $18$, $24$, $33$, $43$, $55$, $69$, $86$, $104$, $126$, $150$, etc.
According to~\cite[A001973]{oeis}, this sequence $a(d)=\delta_{4,d}$ satisfies the recurrence relation
\[
a(d)= 2a(d-1) - a(d-3) - a(d-4) + 2a(d-6) - a(d-7). 
\]
Furthermore, its generating function is 
\[
    H_4(t)=\frac{1+t^3}{(1-t)(1-t^2)^2(1-t^3)}
\]
It is known that $\F[\bfy_4]^x=Z_4$ is generated by 
$z_1$, $z_2$, $z_3$, $z_4$ (as defined in Section~\ref{sec:firstsec}), and 
\[
    \zeta_5=2y_2^3 - 6y_1y_2y_3 + 9y_0y_3^2 + 6y_1^2y_4 - 12y_0y_2y_4;
\] 
that is a generator of degree one, two of degree two and 
two of degree three (see Example~\ref{ex:z14}). 
\end{example}

\begin{example}\label{ex:n5}
    Consider the case $n=5$. Then $\delta_{5,d}$ is equal to the number of partitions of 
    $\lfloor 5d/2\rfloor$ with 
    at most $d$ parts of size at most $5$. The first elements of this sequence are 
    $1$, $1$, $3$, $6$, $12$, $20$, $32$, $49$, $73$, $102$, $141$, $190$, $252$, $325$, $414$, $521$, 
    $649$, $795$, $967$, $1165$, $1394$, etc.
    According to~\cite[A001975]{oeis}, this sequence $a(d)=\delta_{5,d}$ satisfies the recurrence relation
    \begin{align*}
    a(d)&=2a(d-1)-a(d-2)+a(d-4)-2a(d-5)+2a(d-6)-2a(d-7)\\&+2a(d-8)-2a(d-9)+2a(d-11)-2a(d-12)+2a(d-13)\\&-2a(d-14)+
    2a(d-15)-a(d-16)+a(d-18)-2a(d-19)+a(d-20)
    \end{align*}
    Furthermore, its generating function is
    {\small 
    \[
        H_5(t)=\frac{-(t^{14} -t^{13} +2t^{12} +t^{11} +2t^{10} +3t^9 +t^8 +5t^7 +t^6 +3t^5 +2t^4 +t^3 +2t^2 -t+1)}{(t^4+1)(t^2+t+1)(t^2-t+1)(t^2+1)^2(t+1)^3(t-1)^5}.
    \]}
    It is claimed in~\cite{Ooms1,Ooms2} that $Z_5$ is generated by 23 generators.  
\end{example}

Theorem~\ref{th:th3} stated in the Introduction 
gives a closed formula for the Hilbert series $H_n(t)$ for all $n$.
The formula originates from~\cite{almkvist2} where it is derived in a different context 
with proof in~\cite{almkvist1}. We include a proof for completeness and easy reference.

\begin{proof}[The proof of Theorem~\ref{th:th3}]
Suppose that $u_{-n},u_{-n+2},\ldots,u_{n-2},u_n$ are variables such that the weight of 
$u_i$ is $i$. Suppose that $I_n$ denotes the index set $\{-n,-n+2,\ldots,n-2,n\}$. 
We denote by $\boldsymbol{\alpha}$ a vector $(\alpha_{-n},\alpha_{-n+2},\ldots,\alpha_n)$
and by $A_{d}$ the set of $\boldsymbol{\alpha}$ such that $0\leq \alpha_i\leq d$ and 
$\sum_{j\in I_n}\alpha_j=d$.   
The weight of a monomial $\prod_{j\in I_n}u_j^{\alpha_j}$ is $\sum_{j\in I_n}\alpha_j\cdot j$
and its degree is $\sum_{j\in I_n}\alpha_j$. Recall, for $n\geq 0$ and $d\geq 0$, 
 that $\delta_{n,d}$ is the number of monomials of degree $d$ and weight zero plus the number of monomials of degree $d$ and weight one in the $u_i$.  
   We are required to show that the right-hand side 
of~\eqref{eq:inteq} can be written as 
\[ 
    \sum_{d=0}^\infty\delta_{n,d}t^d.
\]

Let us first work the denominator of the function inside the integral. More specifically, we compute that  
    \begin{align*}
        \prod_{k=0}^n\frac{1}{1-t\exp(i(n-2k)\varphi)}=\prod_{k\in I_n}\frac{1}{1-t\exp(ik\varphi)}=
        \prod_{k\in I_n}\sum_{j\geq 0}{\exp(ijk\varphi)t^j}.
    \end{align*}
The product of the power series in the last expression is a power series $\sum_{d\geq 0}c_dt^d$ whose $d$-th coefficient is 
\[
    c_d=\sum_{\boldsymbol{\alpha}\in A_d}\exp\left(i\varphi\sum_{k\in I_n}k\alpha_{k}\right).
\]
The integral in~\eqref{eq:inteq} can be split into two parts:
\begin{equation}\label{eq:inteq2}
    \frac{1}{2\pi}\int_{-\pi}^\pi\frac{1}{\prod_{k=0}^n(1-t\exp(i(n-2k)\varphi))}\,d\varphi+
    \frac{1}{2\pi}\int_{-\pi}^\pi\frac{\exp(i\varphi)}{\prod_{k=0}^n(1-t\exp(i(n-2k)\varphi))}\,d\varphi.
\end{equation}
By the calculations above, the first summand in~\eqref{eq:inteq2} can be written as 
\begin{equation}\label{eq:inteq3}
    \frac{1}{2\pi}\left(\int_{-\pi}^\pi \sum_{d\geq 0}c_d\,d\varphi\right)t^d=\frac 1{2\pi}
    \sum_{d\geq 0}\left(\sum_{\boldsymbol{\alpha}\in A_d}\int_{-\pi}^\pi \exp\left(i\varphi\sum_{k\in I_n}k\alpha_{k}\right)\,d\varphi\right)t^d.
\end{equation} 
We have, for $k\in\Z$, that   
\[
    \int_{-\pi}^\pi\exp(ik\varphi)\,d\varphi=\left\{\begin{array}{ll}0 & \mbox{if $k\neq 0$}\\
        2\pi&\mbox{if $k=0$.}\end{array}\right. 
\]
Hence, among the integrals in the coefficient of $t^d$ in \eqref{eq:inteq3} only those are different from zero
which are taken for  
$\boldsymbol{\alpha}\in A_d$ with $\sum_{k\in I_n}k\alpha_k=0$. Furthermore, for such an $\boldsymbol{\alpha}$, the value of  the integral is $2\pi$ and hence, since there is a $1/(2\pi)$ factor, each such 
$\boldsymbol{\alpha}$ contributes with one to the coefficient of $t^d$.  
Now each such choice of $\boldsymbol{\alpha}$ 
corresponds to a monomial $u_{-n}^{\alpha_{-n}}\cdots u_n^{\alpha_n}$ with degree $d$ and weight zero. 
Thus, the first term in~\eqref{eq:inteq2} can be written as $\sum_{d\geq 0} a_dt^d$ where $a_d$ is the 
number of monomials in the $u_i$ of degree $d$ and weight zero.

Now consider the second summand in~\eqref{eq:inteq2} which is equal 
\begin{equation}\label{eq:inteq4}
    \frac 1{2\pi}\left(\int_{-\pi}^\pi \exp(i\varphi)\sum_{d\geq 0}c_d\,d\varphi\right)t^d=
    \frac{1}{2\pi}
    \sum_{d\geq 0}\left(\sum_{\boldsymbol{\alpha}\in A_d}\int_{-\pi}^\pi \exp\left(i\varphi \left(1+\sum_{k\in I_n}ka_{k}\right)\right)\,d\varphi\right)t^d.
\end{equation} 
For reasons similar to the ones in the previous paragraph, 
among the integrals in the coefficient of $t^d$ the ones that are different from zero are taken for   
$\boldsymbol{\alpha}\in A_d$ with $1+\sum_{k\in I_n}ka_k=0$.
Now each such choice of $\boldsymbol{\alpha}$ 
corresponds to a monomial $u_{-n}^{a_{-n}}\cdots u_n^{a_n}$ with degree $d$ and weight $-1$. 
Thus, the second term in~\eqref{eq:inteq2} can be written as $\sum_{d\geq 0} b_dt^d$ where $b_d$ is the 
number of monomials in the $u_i$ of degree $d$ and weight $-1$ which is the same as the number of monomials 
of degree $d$ and weight one. 
\end{proof}

The first summand in~\eqref{eq:inteq2} can be viewed as the Molien series 
for the invariant ring of the
compact Lie group $S^1=\{\exp(i\varphi)\mid \varphi\in[-\pi,\pi)\}$ inside $\F[\bfy_n]$
where an element $\exp(i\varphi)\in S^1$ acts on the variables $y_0,\ldots,y_n$ by multiplication with the scalars 
\[
    \exp(in\varphi),\ \exp(i(n-2)\varphi),\ldots,\exp(i(-n+2)\varphi),\ \exp(-in\varphi).
\] 
(See~\cite[Theorem~3.4.2]{DK} for Molien series for the invariant ring of finite groups.)
The second term is a translate of the same Molien series; taking such a translate  corresponds to considering monomials of weight one instead of zero.

\subsection{Computation of $\boldsymbol{H_n(t)}$ for higher values of $\boldsymbol n$}
Let $V_{n,d}$ be the space of degree-$d$ polynomials in the variables $y_0,\ldots,y_n$ considered
as an $\slg$-module. That is, $V_{n,1}$ is considered as an irreducible $\slg$-module of dimension $n+1$ with weights 
$n,n-2,\ldots,-n$ and $V_{n,d}$ is isomorphic to the symmetric power $S^d(V_{n,1})$. Then $V_{n,d}$ is spanned by monomials
of the form $y_0^{\alpha_0}\cdots y_n^{\alpha_n}$ with $\alpha_0+\cdots+\alpha_n=d$. 
As above, the weight of such a monomial is  
$n\alpha_0+(n-2)\alpha_1+\cdots-n\alpha_n$. 
Set
\[
    M_n(d,w)=\mbox{the number of monomials in $V_{n,d}$ of weight $w$}
\]
and define the generating function $\mu_n(t,z)$ as 
\[
    \mu_n(t,z)=\sum_{d\geq 0,w}^\infty M_n(d,w)t^dz^w.
\]
Then standard computation with generating functions shows that the following theorem is valid.

\begin{theorem}\label{th:mu}
    For $n\geq 1$, we have 
\begin{equation}\label{eq:mu}
    \mu_n(t,z)= \prod_{k\in\{-n,-n+2,\ldots,n-2,n\}}\frac{1}{1-tz^k}.
\end{equation}
\end{theorem}

\begin{example}
If $n=1$, then 
\[
\mu_1(t,z)=1+t(z^{-1}+z)+t^2(z^{-2}+1+z^2)+t^3(z^{-3}+z^{-1}+z+z^3)+ \cdots
\]
In this case we have to count monomials of even degree and weight zero and monomials of 
odd degree and weight one. These monomials are counted by the terms 
\[
    1+tz+t^2+t^3z+\cdots
\] 
showing that $\delta_{1,d}=1$ for all $d\geq 0$. This agrees with the values of $\delta_{1,d}$ computed 
in Example~\ref{ex:n1}.
\end{example}

\begin{example}
If $n=2$, then 
\begin{align*}
\mu_2(t,z)&=1 + t (z^2 + z^{-2} + 1) + t^2 (z^4 + z^{-4} + z^2 + z^{-2} + 2) \\&+ t^3 (z^6 + z^{-6} + z^4 + z^{-4} + 2 z^2 + 2z^{-2} + 2) \\&+ t^4 (z^8 + z^{-8} + z^6 + z^{-6} + 2 z^4 + 2z^{-4} + 2 z^2 + 2z^{-2} + 3) \\&+ t^5 (z^{10} + z^{-10} + z^8 + z^{-8} + 2 z^6 + 2z^{-6} + 2 z^4 + 2z^{-4} + 3 z^2 + 3z^{-2} + 3) + O(t^6)
\end{align*}
The part of this series expansion containing only the terms with $z^0$ is
\[
    1+t+2t^2+2t^3+3t^4+3t^5+\cdots
\]
and the coefficients agree with the values for $\delta_{2,d}$ computed in Example~\ref{ex:n2}.
\end{example}

\begin{example}\label{ex:zel}
Let us compute $H_3(t)$ using the method outlined by Ekhad and Zeilberger. By Theorem~\ref{th:mu} we have
that 
\[
    \mu_3(t,z)=\frac{1}{(1-z^{-3}t)(1-z^{-1}t)(1-zt)(1-z^{3}t)}=
    z^4\frac{1}{(z^3-t)(z-t)(1-zt)(1-z^{3}t)}.
\]
Using that the factors $z^3-t$, $z-t$, $1-zt$, $1-z^3t$ are coprime polynomials in the variable 
$z$ over the rational
function field 
$\Q(t)$, we have that there are polynomials $Q_1,Q_2,Q_3,Q_4\in \Q(t)[z]$ such that 
\begin{align*}
    \mu_3(t,z)&=z^4\left(\frac{Q_1}{z^3-t}+\frac{Q_2}{z+t}+\frac{Q_3}{1-zt}+\frac{Q_4}{1-z^3t}\right)\\
    &=z\frac{Q_1}{1-z^{-3}t}+z^3\frac{Q_2}{1-z^{-1}t}+z^4\frac{Q_3}{1-zt}+z^4\frac{Q_4}{1-z^3t}.
\end{align*}
The expressions $Q_i$ in the last displayed equation can be determined 
in a computational algebra system, such as Magma~\cite{Magma}, using the Extended Euclidean Algorithm. 
Writing $\mu_3(t,z)=\sum_{k\in\Z} q_k z^k$ where $q_k\in\Q(t)$, 
remembering that $n=3$ is odd, the Hilbert series $H_3(t)$ is determined by the coefficients $q_1$ and 
$q_0$ and they depend only on $Q_1$ and $Q_2$. This way one obtains the same Hilbert series as 
in Example~\ref{ex:n3}. The web page~\cite{zeilberg} contains a report on computations 
that determined the rational expressions for $H_n(t)$ for $n\leq 16$ using partial fraction decomposition of 
$\mu_n(t,z)$ as given in equation~\eqref{eq:mu}; see also the note~\cite{coins} for details.
This procedure was also implemented by the authors of the present paper and the implementation was used 
to verify the rational expressions for $H_n(t)$ in Examples~\ref{ex:n2}, and \ref{ex:n3}--\ref{ex:n5}. 
Using our Magma implementation of the procedure, we also computed the rational expression for $H_n(t)$ 
for $n\leq 18$. The computation for $n=18$ took about 10 minutes on an Intel(R) Core(TM) i5-10210U CPU 
running at 1.6~GHz.  
\end{example}

\section{Prime characteristic}\label{sec:prime}

Throughout this final section, we let $n\geq 1$ fixed, $\F$ denotes a field of characteristic $p$ with $p\geq n+1$, and $\g=\g(n+2)$ is the $(n+2)$-dimensional standard filiform Lie algebra defined over $\F$. 
Note that $U(\g)$ can be viewed as $U_{\Z}(\g)\otimes\F$ where $U_{\Z}(\g)$ is the universal enveloping 
algebra of the standard filiform Lie algebra $\g_\Z(n+2)$ defined over $\Z$. Since the elements 
$z_1,\ldots,z_n$ in Section~\ref{sec:firstsec} are central in $U_\Z(\g)$, they are also central 
in $U(\g)= U_{\Z}(\g)\otimes\F$. In characteristic $p$, however, there are other central elements that do 
not occur in characteristic zero.
Let $Z_p(\g)$ denote the $p$-center of $U(\g)$; that is 
\[
    Z_p(\g)=\F[x^p,y_0,y_1^p,\ldots,y_n^p]. 
\]
The condition that $p\geq n+1$ implies that 
$Z_p(\g)$ is isomorphic to the polynomial algebra $\F[\bft_{n+2}]$ and also that $Z_p(\g)\subseteq Z(\g)$. Let $K(\g)$ and $K_p(\g)$ denote the fraction
fields of $Z(\g)$ and $Z_p(\g)$ respectively. We also let $D(\g)$ denote the division algebra of fractions of $U(\g)$.  Suppose that $z_i$ is the sequence of elements defined in Section~\ref{sec:expgens}.
The leading coefficient of $z_i$ 
is either $2$ or $i$ by Lemma~\ref{lem:leadtermz}, which is nonzero by the assumption that $p\geq n+1$.

In the rest of this section, the elements $z_1,\ldots,z_n$ are the same as defined in Section~\ref{sec:expgens}
and recall that $z_1=y_0\in Z_p(\g)$. 

\begin{theorem}\label{th:gensp}
   We have that  $K(\g)=K_p(\g)(z_2,\ldots,z_{n})$ and $\dim_{K(\g)}D(\g)=p^2$. 
\end{theorem}
\begin{proof}
Consider the chain 
\[ 
   K_p(\g)\subset K(\g)\subset D(\g)
\]
of division algebras (where the first two terms are fields). We have that   
\[
    \dim_{K_p(\g)}D(\g)=p^{\dim \g-\dim C(\g)}=p^{n+1}\quad\mbox{and}\quad 
    \dim_{K(\g)}D(\g)\geq p^2
\] 
(see~\cite[Theorem~3.2]{LdJSch}).
   Furthermore, since $a^p\in K_p(\g)$ for all $a\in K(\g)$, 
  $K(\g)$ is a purely inseparable extension of $K_p(\g)$ of dimension $p^k$ where $k\leq n-1$. 
  Since $z_i\not\in K_p(\g)(z_2,\ldots,z_{i-1})$ but $z_i^p\in K_p(\g)$ for all $i\geq 2$ and
  $K_p(\g)(z_2,\ldots,z_i)$ is a purely inseparable extension of $K_p(\g)(z_2,\ldots,z_{i-1})$, 
  it follows that
  \[
      \dim_{K_p(\g)(z_2,\ldots,z_{i-1})}K_p(\g)(z_2,\ldots,z_{i-1},z_i)=p. 
\]
This implies that $K(\g)=K_p(z_2,\ldots,z_{n})$ and $\dim_{K(\g)}D(\g)=p^2$.
\end{proof}

\begin{proposition}
    Suppose that $R_n=Z_p(\g)[y_0^{-1},z_2,\ldots,z_{n}]$ and, for $i\in\{2,\ldots,n\}$, set 
    $\alpha_i=z_i^p$. Then $\alpha_i\in Z_p(\g)$ and  following are valid.
\begin{enumerate}
    \item $R_n\cong Z_p(\g)[y_0^{-1},t_2,\ldots,t_{n}]/(t_2^p-\alpha_2,\ldots,t_{n}^p-\alpha_{n})$.
    \item $R_n$ is a regular ring.
\end{enumerate}
Consequently, $R_n$ is a normal domain (that is, integrally closed in its field of fractions).
\end{proposition}
\begin{proof} 
The fact that $\alpha_i\in Z_p(\g)$ follows from the fact that $\alpha_i$ is an expression in 
the commuting variables $y_0,y_1,\ldots,y_n$ and that $y_0\in Z_p(\g)$ and, for all $i\geq 1$, $y_i^p\in Z_p(\g)$.

We will prove the  assertions (1)--(2) of the lemma simultaneously by induction on $n$. 
The basis of the induction is the case when $n=1$ and the assertions in this case 
are valid since $R_1=Z_p(\g)[y_0^{-1}]\cong\F[x^p,y_0,y_0^{-1}]$ ($x^p$ and $y_0$ are algebraically independent) 
is a regular domain. Since a regular domain is also normal, we obtain that $R_1$ is a normal 
domain.

Suppose that assertions (1) and (2) hold for the algebra $R_{n-1}$ that corresponds to the 
standard filiform Lie algebra $\g(n+1)$.  
That is,  
\begin{align*}
    R_{n-1}&=\F[y_0^{-1},y_0,y_1^p,\ldots,y_{n}^p,z_2,\ldots,z_{n-1}]\\
    &\cong Z_p(\g(n+1))[y_0^{-1},z_2,\ldots,z_{n-1}]\\&
    \cong Z_p(\g(n+1))[y_0^{-1},t_2,\ldots,t_{n-1}]/(t_2^p-\alpha_2,\ldots,t_{n-1}^p-\alpha_{n-1}).
\end{align*} 
Further, the induction hypothesis also states that $R_{n-1}$ is regular and hence it is normal. 
Since, $y_n^p$ is a free variable over $R_{n-1}$, $R_{n-1}[y_n^p]$ is normal. 
Also note that $\alpha_i\in R_{n-1}[y_n^p]\setminus (R_{n-1}[y_n^p])^p$ and 
$t_{n}^p-\alpha_n$ is prime in $R_{n-1}[y_n^p][t_{n}]$ (see~\cite[Lemma~2.5]{LdJSch}).  
Now since $t_{n}^p-\alpha_n$ is prime, we obtain that 
\begin{align*}
    R_n&=Z_p(\g)[y_0^{-1},z_2,\ldots,z_{n}]=R_{n-1}[y_n^p,z_{n}]\\&\cong R_{n-1}[y_n^p][t_{n}]/(t_{n}^p-\alpha_{n})\\&\cong
    Z_p(\g)[y_0^{-1},t_2,\ldots,t_{n}]/(t_2^p-\alpha_2,\ldots,t_{n}^p-\alpha_{n}). 
\end{align*}
This proves claim~(1).

Let us turn to claim~(2). 
We identify $Z_p(\g)[y_0^{-1},t_2,\ldots,t_{n}]$ with the localization $\F[\bfu_{2n+1}]_{u_2}$ 
at the multiplicative set generated by $u_2$ of the  polynomial ring $\F[\bfu_{2n+1}]$ 
in $2n+1$ variables where $x,y_0,y_1^p,\ldots,y_n^p$ are identified with the variables  
$u_1,u_2,\ldots,u_{n+2}$ and $t_2,\ldots,t_{n}$ are identified with the variables $u_{n+3},\ldots,u_{2n+1}$, 
respectively. In particular, $y_0$ corresponds to the variable $u_2$ and this is why the 
polynomial ring is localized at $u_2$. Suppose, for $i\in\{2,\ldots,n\}$ that $f_i$ is the image of the polynomial 
$t_{i}^p-\alpha_i$ in $\F[\bfu_{2n-2}]_{u_2}$. 
The argument in the previous paragraph implies that 
\[
    R_n\cong\F[\bfu_{2n+1}]_{u_2}/(f_2,\ldots,f_{n}).
\]
For $i=\{2,\ldots,n\}$, let $D_i$ denote the derivation $\partial/\partial u_{i+2}$
of $\F[\bfu_{2n+1}]_{u_2}$. Then we have from Lemma~\ref{lem:leadtermz} that $D_i(f_j)=0$ for all $j<i$ and 
$D_i(f_i)=c_iu_2^{k_i}$ where $k_i$ and $c_i$ are positive integers and $c_i<p$. Hence 
the Jacobian matrix $(D_if_j)$ is upper triangular and $\det (D_if_j)=cu_2^k$ where 
$c,k\in\mathbb N$ and $p\nmid c$. 


Suppose that $\bar{P}\in \mbox{Spec}(R)$. Then $\bar{P}=P/\left(f_{2},\ldots, f_{n}\right)$ 
for some prime ideal $P\in \mbox{Spec}(\F[\bfu_{2n+1}]_{u_2})$ 
such that  $u_{2}\not\in P$ and $(f_{2},\ldots, f_{n})\subseteq P$. 
Thus $\det(D_{i}f_{j})\not\in P$, otherwise $u_{2}\in P$ would be true. 
The Jacobian Criterion for Regularity~\cite[Theorem 30.4]{matsumura86} implies that 
 $\F[\bfu_{2n+1}]_{P}/\left(f_{2},\ldots, f_{n}\right)_{P}$ is a regular ring; 
 that is $(R_n)_{\bar{P}}$ is a regular ring. This shows that $R_n$ is regular. Since a regular ring is normal, $R_n$ is also a normal domain.
\end{proof}

\begin{proof}[The proof of Theorem~\ref{th:th4}]
Set $R=Z_p(\g)[z_2,\ldots,z_{n}]$ and $R[y_0^{-1}]=Z_p(\g)[y_0^{-1},z_2,\ldots,z_{n}]$.
Theorem~\ref{th:gensp} implies that 
\[
    K(\g)=K_p(\g)(z_2,\ldots,z_{n})
\]
and so the fraction fields of $R$, $R[y_0^{-1}]$ and $Z(\g)$ coincide. Let us denote this field 
by $K$. For domains $A\subseteq B$, let $A^B$ denote the integral closure of $A$ in $B$. 
Since $R[y_0^{-1}]$ is integrally closed,
\[
    R^K\subseteq R[y_0^{-1}]^K=R[y_0^{-1}].
\]
This means that the integral closure $R^K$ must be contained in $R[y_0^{-1}]$ and hence 
$R^K=R^{R[y_0^{-1}]}$. 
For the other statement, recall that $Z(\g)$ is an integral extension of $R$ in $K$ and 
$Z(\g)$ is integrally closed in $K$ (by Zassenhaus' Theorem~\cite{zassenhaus53}). Thus  
\[
    Z(\g)\subseteq R^K\subseteq Z(\g)^K=Z(\g). 
\]
It follows that equality must hold in each step of the previous chain which implies that 
$Z(\g)=R^K$ as claimed.
\end{proof}



\appendix
\allowdisplaybreaks
\section{The generating sets of $Z_5$, $Z_6$, $Z_8$}\label{app:gens}
We computed a minimal generating set for $Z_5$ of this form consisting of the following 23 polynomials.

%

{\small
\begin{align*}
    \mbox{degree 1: }z_{1}&=y_0;\\
   \mbox{degree 2: }z_{2}&=y_0\circ_{2}y_0;
   z_{3}=y_0\circ_{4}y_0;\\
   \mbox{degree 3: }z_{4}&=y_0\circ_{2}y_0\circ_{1}y_0;
   z_{5}=y_0\circ_{4}y_0\circ_{1}y_0;
   z_{6}=y_0\circ_{4}y_0\circ_{2}y_0;\\
   \mbox{degree 4: }z_{7}&=y_0\circ_{4}y_0\circ_{2}y_0\circ_{1}y_0;
   z_{8}=y_0\circ_{4}y_0\circ_{2}y_0\circ_{2}y_0;
   z_{9}=y_0\circ_{4}y_0\circ_{1}y_0\circ_{5}y_0;\\
   \mbox{degree 5: }z_{10}&=y_0\circ_{4}y_0\circ_{2}y_0\circ_{2}y_0\circ_{1}y_0;
   z_{11}=z_{3}^{2}\circ_{3}y_0;
   z_{12}=z_{3}^{2}\circ_{4}y_0;\\
   \mbox{degree 6: }z_{13}&=z_{3}^{2}\circ_{4}y_0\circ_{1}y_0;
   z_{14}=z_{3}^{2}\circ_{3}y_0\circ_{3}y_0;\\
   \mbox{degree 7: }z_{15}&=z_{3}^{2}\circ_{3}y_0\circ_{3}y_0\circ_{1}y_0;
   z_{16}=z_{3}^{2}\circ_{4}y_0\circ_{1}y_0\circ_{4}y_0;\\
   \mbox{degree 8: }z_{17}&=z_{3}^{}z_{12}^{}\circ_{3}y_0;
   z_{18}=z_{3}^{}z_{11}^{}\circ_{5}y_0;\\
   \mbox{degree 9: }z_{19}&=z_{3}^{}z_{12}^{}\circ_{3}y_0\circ_{2}y_0;\\
   \mbox{degree 11: }z_{20}&=z_{3}^{}z_{17}^{}\circ_{4}y_0;\\
   \mbox{degree 12: }z_{21}&=z_{3}^{2}z_{16}^{}\circ_{5}y_0;\\
   \mbox{degree 13: }z_{22}&=z_{3}^{}z_{12}^{2}\circ_{4}y_0;\\
   \mbox{degree 18: }z_{23}&=z_{3}^{2}z_{22}^{}\circ_{5}y_0;\\
   \end{align*}
}
The minimal generating system for $Z_6$ consists of the following 26 polynomials:
{\small
\begin{align*}
    \mbox{degree 1: }z_{1}&=y_0;\\
   \mbox{degree 2: }z_{2}&=y_0\circ_{2}y_0;
   z_{3}=y_0\circ_{4}y_0;
   z_{4}=y_0\circ_{6}y_0;\\
   \mbox{degree 3: }z_{5}&=y_0\circ_{2}y_0\circ_{1}y_0;
   z_{6}=y_0\circ_{4}y_0\circ_{1}y_0;
   z_{7}=y_0\circ_{4}y_0\circ_{2}y_0;
   z_{8}=y_0\circ_{4}y_0\circ_{4}y_0;\\
   \mbox{degree 4: }z_{9}&=y_0\circ_{4}y_0\circ_{2}y_0\circ_{1}y_0;
   z_{10}=y_0\circ_{4}y_0\circ_{4}y_0\circ_{1}y_0;
   z_{11}=y_0\circ_{4}y_0\circ_{4}y_0\circ_{2}y_0;\\
   z_{12}&=y_0\circ_{4}y_0\circ_{2}y_0\circ_{6}y_0;\\
   \mbox{degree 5: }z_{13}&=y_0\circ_{4}y_0\circ_{4}y_0\circ_{2}y_0\circ_{1}y_0\
   z_{14}=y_0\circ_{4}y_0\circ_{4}y_0\circ_{2}y_0\circ_{3}y_0;\\
   z_{15}&=y_0\circ_{4}y_0\circ_{4}y_0\circ_{2}y_0\circ_{4}y_0;\\
   \mbox{degree 6: }z_{16}&=y_0\circ_{4}y_0\circ_{4}y_0\circ_{2}y_0\circ_{4}y_0\circ_{1}y_0;
   z_{17}=y_0\circ_{4}y_0\circ_{4}y_0\circ_{2}y_0\circ_{3}y_0\circ_{2}y_0;\\
   z_{18}&=z_{3}^{}z_{8}^{}\circ_{6}y_0;\\
   \mbox{degree 7: }z_{19}&=z_{8}^{2}\circ_{3}y_0;
   z_{20}=z_{8}^{2}\circ_{4}y_0;\\
   \mbox{degree 8: }z_{21}&=z_{8}^{2}\circ_{3}y_0\circ_{4}y_0;\\
   \mbox{degree 9: }z_{22}&=z_{8}^{2}\circ_{3}y_0\circ_{4}y_0\circ_{2}y_0;\\
   \mbox{degree 10: }z_{23}&=z_{8}^{2}\circ_{3}y_0\circ_{4}y_0\circ_{2}y_0\circ_{4}y_0;
   z_{24}=z_{8}^{3}\circ_{6}y_0;\\
   \mbox{degree 12: }z_{25}&=z_{8}^{}z_{21}^{}\circ_{4}y_0;\\
   \mbox{degree 15: }z_{26}&=z_{8}^{2}z_{21}^{}\circ_{6}y_0;\\
   \end{align*}   
}

The generators of $Z_8$ are the following polynomials.
{\small
\begin{align*}
    \mbox{Degree 1: }z_{1}&=y_0;\\
   \mbox{Degree 2 :}z_{2}&=y_0\circ_{2}y_0;
   z_{3}=y_0\circ_{4}y_0;
   z_{4}=y_0\circ_{6}y_0;
   z_{5}=y_0\circ_{8}y_0;\\
   \mbox{Degree 3 :}z_{6}&=y_0\circ_{2}y_0\circ_{1}y_0;
   z_{7}=y_0\circ_{4}y_0\circ_{1}y_0;
   z_{8}=y_0\circ_{6}y_0\circ_{1}y_0;
   z_{9}=y_0\circ_{4}y_0\circ_{2}y_0;\\
   z_{10}&=y_0\circ_{6}y_0\circ_{2}y_0;
   z_{11}=y_0\circ_{6}y_0\circ_{3}y_0;
   z_{12}=y_0\circ_{6}y_0\circ_{4}y_0;
   z_{13}=y_0\circ_{4}y_0\circ_{8}y_0;\\
   \mbox{Degree 4 :}z_{14}&=y_0\circ_{4}y_0\circ_{2}y_0\circ_{1}y_0;
   z_{15}=y_0\circ_{6}y_0\circ_{2}y_0\circ_{1}y_0;
   z_{16}=y_0\circ_{6}y_0\circ_{3}y_0\circ_{1}y_0;\\
   z_{17}&=y_0\circ_{6}y_0\circ_{4}y_0\circ_{1}y_0;
   z_{18}=y_0\circ_{6}y_0\circ_{3}y_0\circ_{2}y_0;
   z_{19}=y_0\circ_{6}y_0\circ_{4}y_0\circ_{2}y_0;\\
   z_{20}&=y_0\circ_{6}y_0\circ_{4}y_0\circ_{3}y_0;
   z_{21}=y_0\circ_{6}y_0\circ_{4}y_0\circ_{4}y_0;
   z_{22}=y_0\circ_{6}y_0\circ_{3}y_0\circ_{5}y_0;\\
   z_{23}&=y_0\circ_{6}y_0\circ_{2}y_0\circ_{8}y_0;\\
   \mbox{Degree 5 :}z_{24}&=y_0\circ_{6}y_0\circ_{4}y_0\circ_{2}y_0\circ_{1}y_0;
   z_{25}=y_0\circ_{6}y_0\circ_{4}y_0\circ_{4}y_0\circ_{1}y_0;\\
   z_{26}&=y_0\circ_{6}y_0\circ_{3}y_0\circ_{5}y_0\circ_{1}y_0;
   z_{27}=y_0\circ_{6}y_0\circ_{4}y_0\circ_{3}y_0\circ_{2}y_0;\\
   z_{28}&=y_0\circ_{6}y_0\circ_{4}y_0\circ_{4}y_0\circ_{2}y_0;
   z_{29}=y_0\circ_{6}y_0\circ_{4}y_0\circ_{4}y_0\circ_{3}y_0;\\
   z_{30}&=y_0\circ_{6}y_0\circ_{3}y_0\circ_{5}y_0\circ_{3}y_0;
   z_{31}=y_0\circ_{6}y_0\circ_{4}y_0\circ_{4}y_0\circ_{4}y_0;\\
   z_{32}&=y_0\circ_{6}y_0\circ_{3}y_0\circ_{5}y_0\circ_{4}y_0;
   z_{33}=y_0\circ_{6}y_0\circ_{4}y_0\circ_{3}y_0\circ_{6}y_0;\\
   z_{34}&=z_{4}^{2}\circ_{8}y_0;\\
   \mbox{Degree 6 :}z_{35}&=y_0\circ_{6}y_0\circ_{4}y_0\circ_{3}y_0\circ_{6}y_0\circ_{1}y_0;
   z_{36}=y_0\circ_{6}y_0\circ_{4}y_0\circ_{4}y_0\circ_{4}y_0\circ_{1}y_0;\\
   z_{37}&=y_0\circ_{6}y_0\circ_{4}y_0\circ_{3}y_0\circ_{6}y_0\circ_{2}y_0;
   z_{38}=y_0\circ_{6}y_0\circ_{4}y_0\circ_{4}y_0\circ_{4}y_0\circ_{3}y_0;\\
   z_{39}&=y_0\circ_{6}y_0\circ_{3}y_0\circ_{5}y_0\circ_{4}y_0\circ_{3}y_0;
   z_{40}=y_0\circ_{6}y_0\circ_{4}y_0\circ_{4}y_0\circ_{4}y_0\circ_{4}y_0;\\
   z_{41}&=y_0\circ_{6}y_0\circ_{3}y_0\circ_{5}y_0\circ_{4}y_0\circ_{4}y_0;
   z_{42}=y_0\circ_{6}y_0\circ_{3}y_0\circ_{5}y_0\circ_{3}y_0\circ_{6}y_0;\\
   z_{43}&=z_{4}z_{12}\circ_{8}y_0;\\
   \mbox{Degree 7 :}z_{44}&=y_0\circ_{6}y_0\circ_{3}y_0\circ_{5}y_0\circ_{3}y_0\circ_{6}y_0\circ_{2}y_0;
   z_{45}=y_0\circ_{6}y_0\circ_{4}y_0\circ_{4}y_0\circ_{4}y_0\circ_{4}y_0\circ_{3}y_0;\\
   z_{46}&=y_0\circ_{6}y_0\circ_{3}y_0\circ_{5}y_0\circ_{4}y_0\circ_{4}y_0\circ_{3}y_0;
   z_{47}=y_0\circ_{6}y_0\circ_{3}y_0\circ_{5}y_0\circ_{4}y_0\circ_{4}y_0\circ_{4}y_0;\\
   z_{48}&=y_0\circ_{6}y_0\circ_{4}y_0\circ_{3}y_0\circ_{6}y_0\circ_{2}y_0\circ_{5}y_0;
   z_{49}=y_0\circ_{6}y_0\circ_{4}y_0\circ_{3}y_0\circ_{6}y_0\circ_{2}y_0\circ_{6}y_0;\\
   z_{50}&=y_0\circ_{6}y_0\circ_{3}y_0\circ_{5}y_0\circ_{4}y_0\circ_{3}y_0\circ_{6}y_0;
   z_{51}=z_{12}^{2}\circ_{8}y_0;\\
   \mbox{Degree 8 :}z_{52}&=y_0\circ_{6}y_0\circ_{3}y_0\circ_{5}y_0\circ_{4}y_0\circ_{3}y_0\circ_{6}y_0\circ_{2}y_0;\\
   z_{53}&=y_0\circ_{6}y_0\circ_{4}y_0\circ_{3}y_0\circ_{6}y_0\circ_{2}y_0\circ_{6}y_0\circ_{2}y_0;\\
   z_{54}&=y_0\circ_{6}y_0\circ_{4}y_0\circ_{3}y_0\circ_{6}y_0\circ_{2}y_0\circ_{5}y_0\circ_{4}y_0;\\
   z_{55}&=y_0\circ_{6}y_0\circ_{3}y_0\circ_{5}y_0\circ_{4}y_0\circ_{4}y_0\circ_{4}y_0\circ_{4}y_0;\\
   z_{56}&=z_{4}z_{33}\circ_{6}y_0;
   z_{57}=y_0\circ_{6}y_0\circ_{3}y_0\circ_{5}y_0\circ_{3}y_0\circ_{6}y_0\circ_{2}y_0\circ_{6}y_0;
   z_{58}=z_{4}z_{32}\circ_{8}y_0;\\
   \mbox{Degree 9 :}z_{59}&=y_0\circ_{6}y_0\circ_{4}y_0\circ_{3}y_0\circ_{6}y_0\circ_{2}y_0\circ_{5}y_0\circ_{4}y_0\circ_{4}y_0;\\
   z_{60}&=z_{4}z_{42}\circ_{6}y_0; 
   z_{61}=z_{12}z_{33}\circ_{6}y_0;\\
   z_{62}&=y_0\circ_{6}y_0\circ_{4}y_0\circ_{3}y_0\circ_{6}y_0\circ_{2}y_0\circ_{6}y_0\circ_{2}y_0\circ_{6}y_0;
   z_{63}=z_{4}z_{41}\circ_{8}y_0;\\
   \mbox{Degree 10 :}z_{64}&=z_{4}z_{50}\circ_{6}y_0;
   z_{65}=z_{4}z_{49}\circ_{6}y_0;
   z_{66}=z_{4}z_{48}\circ_{8}y_0;\\
   \mbox{Degree 11 :}z_{67}&=z_{4}z_{56}\circ_{6}y_0;
   z_{68}=z_{4}z_{57}\circ_{6}y_0;\\
   \mbox{Degree 12 :}z_{69}&=z_{4}z_{60}\circ_{6}y_0;\\
   \end{align*}
}

In all of these computations, we used the information given in~\cite{freud} concerning the maximum degree of the polynomials in a minimal generating set. The degree-15 polynomial in the generating set for $Z_6$ has $1370$ nonzero terms, while the degree-12 polynomial in the generating set of $Z_8$ has $3651$ 
nonzero terms.
\end{document}